\documentclass[11pt]{amsart}

\usepackage{amscd,amssymb,amsmath,graphicx,verbatim}
\usepackage{wasysym}
\usepackage{hyperref}
\usepackage[TS1,OT1,T1]{fontenc}
\usepackage{lscape} 
\usepackage{pslatex} 
\usepackage{tikz}
\usetikzlibrary{shapes,arrows}
\usepackage[all]{xy}
\newtheorem{theorem}{Theorem}[section]
\newtheorem{lemma}[theorem]{Lemma}
\newtheorem{corollary}[theorem]{Corollary}
\newtheorem{proposition}[theorem]{Proposition}

\theoremstyle{definition}
\newtheorem{definition}[theorem]{Definition}

\newtheorem{example}[theorem]{Example}

\theoremstyle{remark}
\newtheorem{remark}[theorem]{Remark}

\newcommand{\Dcal}{\ensuremath{\mathcal{D}}}
\newcommand{\Xcal}{\ensuremath{\mathcal{X}}}

\newcommand{\Tcal}{\ensuremath{\mathcal{T}}}
\newcommand{\Fcal}{\ensuremath{\mathcal{F}}}

\newcommand{\Acal}{\ensuremath{\mathcal{A}}}
\newcommand{\Bcal}{\ensuremath{\mathcal{B}}}

\newcommand{\Scal}{\ensuremath{\mathcal{S}}}

\newcommand{\Rcal}{\ensuremath{\mathcal{R}}}
\newcommand{\Lcal}{\ensuremath{\mathcal{L}}}

\newcommand{\Kcal}{\ensuremath{\mathcal{K}}}

\newcommand{\Kbb}{\mathbb{K}}
\newcommand{\Z}{\mathbb{Z}}

\newcommand{\ra}{\rightarrow}

\DeclareMathOperator{\Spec}{Spec}

\numberwithin{equation}{section}

\begin{document}
\title{Universal localisations via silting}
\author{Frederik Marks and Jan {\v S}{\v{t}}\!ov{\'{\i}}{\v{c}}ek}
\address{Frederik Marks, University of Stuttgart, Institute for Algebra and Number Theory, Pfaffenwaldring 57, 70569 Stuttgart, Germany}
\email{marks@mathematik.uni-stuttgart.de}

\address{Jan {\v S}{\v{t}}\!ov{\'{\i}}{\v{c}}ek, Charles University in Prague, Faculty of Mathematics and Physics, Department of Algebra, Sokolovsk\'a 83, 186 75 Praha, Czech Republic}
\email{stovicek@karlin.mff.cuni.cz}

\thanks{The first named author was supported by a grant within the DAAD P.R.I.M.E. program. The second named author was supported by grant GA~\v{C}R 14-15479S from the Czech Science Foundation.} 

\begin{abstract}
We show that silting modules are closely related with localisations of rings. More precisely, every partial silting module gives rise to a localisation at a set of maps between countably generated projective modules and, conversely, every universal localisation, in the sense of Cohn and Schofield, arises in this way. To establish these results, we further explore the finite-type classification of tilting classes and we use the morphism category to translate silting modules into tilting objects. In particular, we prove that silting modules are of finite type.
\end{abstract}
\maketitle


\section{Introduction}
The aim of this article is to show that silting theory provides a powerful tool to study localisations of rings. We will be particularly interested in the concept of universal localisation as introduced by Cohn and Schofield (see \cite{Sch0}). These localisations have proved to be useful in different branches of mathematics like algebraic K-theory (see \cite{NR}), representation theory (see \cite{AA,AS2}) and topology (see \cite{R}). However, in general, very little is known about their homological properties and classification results are only available in special cases (see \cite{KS,Sch2}). 

A general and powerful approach to learn about properties of a ring is to study its representation theory. Here, we show that the notion of silting module provides a useful module-theoretic counterpart of universal localisations.
Our main result states that every universal localisation is controlled by a (possibly large) partial silting module.
As a consequence, the localised ring will be isomorphic to an idempotent quotient of the endomorphism ring of such a module. 
Furthermore, the connection with silting modules will make it possible to better understand universal localisations by using well-established tools in representation theory.

\medskip

The notion of silting module was introduced in \cite{AMV} to provide a common setup to study simultaneously (possibly large) 1-tilting modules over any ring and support $\tau$-tilting modules over a finite dimensional algebra. Silting modules can be understood as the module-theoretic counterpart of two-term silting complexes and they parametrise certain torsion pairs in the module category and in its derived category. 

The fact that a (possibly large) partial 1-tilting module always induces a certain epimorphism in the category of rings goes back to~\cite{CTT}. There, it was shown that the perpendicular category to such a module can be identified with the category of modules over the codomain of a ring epimorphism (see also \cite{GdP}). In the context of finite dimensional algebras over a field, an analogous abelian category associated with a $\tau$-rigid (equivalently, a partial silting) module was studied in \cite{J}. Such an abelian category can be interpreted as the category of modules over a suitable universal localisation. This idea was further developed in \cite{AMV2} through a systematic study of ring epimorphisms arising from partial silting modules. Building on the works above, it was proved that over hereditary rings (see \cite{AMV2}) and over certain finite dimensional algebras (see \cite{MS}) silting modules can be used to classify all universal localisations. Here, we show that this phenomenon remains conceptually true when working over arbitrary rings. However, different techniques are needed.

\medskip

A key ingredient for our approach is the finite-type characterisation of tilting classes which states that a subcategory of modules is a tilting class if and only if it is the class of modules Ext-orthogonal to a set of finitely presented modules of bounded projective dimension. This result was obtained in a series of papers by several authors (see \cite{AC,AHT,BET,BH,BS,ST}). For 1-tilting modules, we establish a similar description of Hom-Ext-orthogonal subcategories (see Theorem \ref{Main 1}). This result directly links universal localisations to tilting theory. In a second step, we use the morphism category associated with a given ring to translate silting modules into tilting objects. As a consequence, we obtain a finite-type characterisation of silting classes (see Theorem \ref{main silting classes}). Moreover, by combining the above ingredients we prove that every ring epimorphism arising from a partial silting module can be interpreted as a localisation at a set of maps between countably generated projective modules (see Theorem \ref{silting countable type}). Conversely, we show that every universal localisation arises from a partial silting module in this way (see Theorem \ref{Main 2}).

\medskip

The paper is organised as follows. After setting up some notation, Section \ref{section tilting} is dedicated to the interplay of tilting modules and cotorsion pairs. We discuss the finite-type characterisation of tilting classes and prove Theorem \ref{Main 1}. In Section \ref{section loc}, we introduce two kinds of localisations, namely tilting ring epimorphisms and universal localisations. Section \ref{section mor} focuses on the morphism category and contains two key lemmas which allow us to translate silting modules into tilting objects. Finally, in Section \ref{section silting}, we introduce silting modules and silting ring epimorphisms and we benefit from the previous work by proving our main theorems.


\section{Notation}\label{setup}
Throughout, let $A$ be a ring with unit. By $Mod(A)$ we denote the category of all left $A$-modules. If not stated otherwise, by an $A$-module we always mean a left $A$-module. The category of all (respectively, all finitely generated) projective $A$-modules is denoted by $Proj(A)$ (respectively, $proj(A)$). By $\Kcal^b(A)$ we denote the bounded homotopy category of chain complexes in $Mod(A)$.
For an $A$-module $X$, we denote by $Add(X)$ (respectively, $Gen(X)$) the full subcategory of $Mod(A)$ containing all direct summands (respectively, all epimorphic images) of direct sums of copies of $X$. For a class $\Xcal$ of $A$-modules, $\Xcal^{\perp_0}$ (respectively, $\Xcal^{\perp_1}$) is defined to be the full subcategory of $Mod(A)$ consisting of all $A$-modules $Y$ such that $Hom_A(X,Y)=0$ (respectively, $Ext_A^1(X,Y)=0$) for all $X\in\Xcal$. Moreover, we set $\Xcal^\perp:=\Xcal^{\perp_0}\cap\Xcal^{\perp_1}$; this category is called the right \textbf{perpendicular category} to $\Xcal$ in \cite{GL}. Dually, we define the subcategories $^{\perp_0}\Xcal, ^{\perp_1}\Xcal$ and $^\perp\Xcal$. We further say that an $A$-module $Y$ is $\Xcal$-\textbf{filtered} if there is an ordinal $\lambda$ and an increasing sequence of submodules of $Y$, $(Y_\alpha\mid \alpha\le\lambda)$, such that $Y_0=0$, $Y=Y_\lambda$, $Y_\alpha=\bigcup_{\beta<\alpha}Y_\beta$ for all limit ordinals $\alpha\le\lambda$ and $Y_{\alpha +1}/Y_\alpha$ is isomorphic to a module in $\Xcal$ for all $\alpha<\lambda$.


\section{Tilting modules and cotorsion pairs}\label{section tilting}
In this article, by a (partial) tilting module, we always mean an a priori large (partial) tilting module of projective dimension at most one.
Recall that an $A$-module $T$ is called \textbf{partial tilting}, if $T^{\perp_1}$ is a torsion class containing $T$ (hence in particular $Gen(T) \subseteq T^{\perp_1}$) and it is called \textbf{tilting}, if $T^{\perp_1}=Gen(T)$.
It follows directly from the definition (see \cite{CT}) that a (partial) tilting module is of projective dimension at most one and that an $A$-module $T$ is tilting if and only if the following three conditions are satisfied:
\begin{itemize}
\item $pd_A(T)\leq 1$;
\item $Ext^1_A(T,T^{(I)})=0$ for all sets $I$; 
\item there are $T_0,T_1\in Add(T)$ and a short exact sequence $$0\ra A\ra T_0\ra T_1\ra 0.$$
\end{itemize}

By \cite[Theorem 1.9]{CT}, every partial tilting module $T$ can be completed to a tilting module $T\oplus T'$ with the same associated torsion class $T^{\perp_1}=(T\oplus T')^{\perp_1}$.
If $T$ is partial tilting, we say that $T^{\perp_1}$ is a \textbf{tilting class}. Tilting classes can be characterised by a finite-type condition. The following result (that can be stated more generally in the context of $n$-tilting modules for a non-negative integer $n$, see \cite{BS,ST}) was proved in a series of articles by different authors.

\begin{theorem}[\cite{AC,AHT,BET,BH}]\label{tilting finite type}
Let $\Tcal$ be a full subcategory of $Mod(A)$. Then $\Tcal$ is a tilting class if and only if there is a set $\Scal$ of finitely presented $A$-modules of projective dimension at most one such that $\Scal^{\perp_1}=\Tcal$. 
\end{theorem}

\begin{proof}
We refer to~\cite[Theorem 2.2]{AHT} and~\cite[Theorem 2.6]{BH}.
\end{proof}

To better understand this characterisation we recall the following definition.

\begin{definition}\label{def cotorsion pairs}
A pair $(\Acal,\Bcal)$ of full subcategories in $Mod(A)$ is called a \textbf{cotorsion pair} if $\Acal^{\perp_1}=\Bcal$ and $^{\perp_1}\Bcal=\Acal$. Moreover, we say that $(\Acal,\Bcal)$ is \textbf{complete} if for all $X$ in $Mod(A)$ there are short exact sequences
$$\xymatrix{0\ar[r] & X\ar[r] & B_X\ar[r] & A_X\ar[r] & 0 }$$
$$\xymatrix{0\ar[r] & B^X\ar[r] & A^X\ar[r] & X\ar[r] & 0 }$$
with $A_X,A^X\in\Acal$ and $B_X,B^X\in\Bcal$.
\end{definition}

For a class $\Lcal$ of $A$-modules, there is an associated cotorsion pair $(^{\perp_1}(\Lcal^{\perp_1}),\Lcal^{\perp_1})$ generated by $\Lcal$. If $\Lcal$ is a set, we get more. The treatment of cotorsion pairs generated by a set can be traced back to \cite{ET}, and here we collect essentials of the theory which we need for our paper.

\begin{proposition}\label{prop cotorsion pairs}
If $\Lcal$ is a set of $A$-modules, then the cotorsion pair $(\Acal,\Bcal):=(^{\perp_1}(\Lcal^{\perp_1}),\Lcal^{\perp_1})$ is complete and the module $A_X$ in Definition~\ref{def cotorsion pairs} can be chosen to be $\Lcal$-filtered. The class $^{\perp_1}(\Lcal^{\perp_1})$ contains all direct summands of $\Lcal$-filtered modules and, if $\Lcal$ contains the regular module $A$, $^{\perp_1}(\Lcal^{\perp_1})$ consists precisely of direct summands of $\Lcal$-filtered modules.
\end{proposition}

\begin{proof}
See~\cite[Theorem 3.2.1 and Corollaries 3.2.3 and 3.2.4]{GT}.
\end{proof}

Now take a tilting module $T$ and consider the cotorsion pair $(^{\perp_1}(T^{\perp_1}),T^{\perp_1})$. Then $^{\perp_1}(T^{\perp_1})\cap T^{\perp_1}=Add(T)$ by \cite[Lemma 2.4]{AC}. We can define the set $\Scal$ from Theorem \ref{tilting finite type} to contain precisely the finitely presented $A$-modules from $^{\perp_1}(T^{\perp_1})$. In fact, by Proposition \ref{prop cotorsion pairs}, all modules in $\Scal$ have projective dimension at most one (since the class of all modules of projective dimension at most one is itself closed under filtrations) and, by the main result in \cite{BH}, we have $\Scal^{\perp_1}=T^{\perp_1}$ (see also \cite{BS}). Conversely, starting with a set $\Scal$ of finitely presented $A$-modules of projective dimension at most one, we consider the complete cotorsion pair $(^{\perp_1}(\Scal^{\perp_1}),\Scal^{\perp_1})$. It follows that the class $\Scal^{\perp_1}$ is closed for coproducts in $Mod(A)$ and that the class $^{\perp_1}(\Scal^{\perp_1})$ contains only $A$-modules of projective dimension at most one. Thus, $\Scal^{\perp_1}$ is a tilting class by \cite[Theorem 4.1]{AC} (see also~\cite[Theorem 2.2]{AHT}).

In this section, we are interested in studying Hom-Ext-orthogonal subcategories to partial tilting modules and to sets of modules of projective dimension at most one. In doing so, we obtain abelian categories that will be of interest in the forthcoming parts of the paper. We prove the following main result.

\begin{theorem}\label{Main 1}
Let $A$ be a ring.
\begin{enumerate}
\item Let $\Scal$ be a set of finitely presented $A$-modules of projective dimension at most one. Then there is a partial tilting $A$-module $T_1$ such that $\Scal^{\perp_1}=T_1^{\perp_1}$ and $\Scal^\perp=T_1^\perp$.
\item Let $T_1$ be a partial tilting $A$-module. Then there is a set $\Scal$ of countably presented $A$-modules of projective dimension at most one such that $\Scal^{\perp_1}=T_1^{\perp_1}$ and $\Scal^\perp=T_1^\perp$.
\end{enumerate}
\end{theorem}

\begin{proof}
$(1)$ Consider the complete tilting cotorsion pair $(^{\perp_1}(\Scal^{\perp_1}),\Scal^{\perp_1})$. For every object $S\in\Scal$, there is a short exact approximation sequence of the form
$$\xymatrix{0\ar[r] & S\ar[r] & \nabla_S\ar[r] & \nabla_S/S\ar[r] & 0}$$
where $\nabla_S$ belongs to $\Scal^{\perp_1}$ and $\nabla_S/S$ is $\Scal$-filtered. In particular, $\nabla_S$ is also $\Scal$-filtered and, hence, $\nabla_S$ lies in the intersection $^{\perp_1}(\Scal^{\perp_1})\cap \Scal^{\perp_1}=Add(T)$ where $T$ is tilting with $Gen(T)=T^{\perp_1}=\Scal^{\perp_1}$. Define $T_1$ to be $\bigoplus_{S\in\Scal}\nabla_S$. It follows that $\Scal^{\perp_1}\subseteq T_1^{\perp_1}$.
We show that $\Scal^{\perp_1}\supseteq T_1^{\perp_1}$ which among others implies that $T_1$ is partial tilting. By applying the functor $Hom_A(-,X)$ for $X\in T_1^{\perp_1}$ to the above exact sequence, we get
$$\xymatrix{\cdots \ar[r] & Ext_A^1(\nabla_S,X)\ar[r] & Ext_A^1(S,X)\ar[r] & Ext_A^2(\nabla_S/S,X)\ar[r] & \cdots}$$
By assumption, we have $Ext_A^1(\nabla_S,X)=0$ and, moreover, since $\nabla_S/S$ is $\Scal$-filtered and all the modules in $\Scal$ have projective dimension at most one, it follows that also $Ext_A^2(\nabla_S/S,X)=0$. Hence, we get $Ext_A^1(S,X)=0$ and, therefore, $\Scal^{\perp_1}=T_1^{\perp_1}$. 

It remains to prove that $\Scal^\perp = T_1^\perp$. We can use an argument analogous to the proof of~\cite[Theorem 3.2(3)]{AA}.
First, let $X$ be in $\Scal^\perp$. Since, by construction, $T_1$ is $\Scal$-filtered, we can choose a particular filtration $(Y_\alpha\mid \alpha\le\lambda)$ with $T_1 = Y_\lambda$. An easy transfinite induction argument shows that $Hom_A(Y_\alpha,X) = 0$ for each $\alpha\le\lambda$. Indeed, this is clear for $\alpha = 0$ and for limit ordinals we use that
$$Hom_A(Y_\alpha,X) = Hom_A(\varinjlim_{\beta<\alpha} Y_\beta,X)\cong\varprojlim_{\beta<\alpha} Hom_A(Y_\beta,X) = 0.$$ 
If $\alpha = \beta+1$ is an ordinal successor, then $Hom_A(Y_\alpha,X) = 0$ since $Hom_A(Y_\beta,X) = 0$ by the inductive hypothesis and $Y_\alpha/Y_\beta \in \Scal$. Applying the conclusion to $T_1 = Y_\lambda$ we obtain $X \in T_1^\perp$.

Conversely, take $X$ in $T_1^\perp\subseteq\Scal^{\perp_1}$. For every $S$ in $\Scal$ we get an exact sequence
$$\xymatrix{\cdots \ar[r] & Hom_A(\nabla_S,X)\ar[r] & Hom_A(S,X)\ar[r] & Ext_A^1(\nabla_S/S,X)\ar[r] & \cdots}$$
By assumption, we have $Hom_A(\nabla_S,X)=0$ and, moreover, since $\nabla_S/S$ is $\Scal$-filtered and $X\in \Scal^{\perp_1}$, it follows from Proposition~\ref{prop cotorsion pairs} that $Ext_A^1(\nabla_S/S,X)=0$. Hence, we get $Hom_A(S,X)=0$ and, thus, $X\in\Scal^{\perp}$. This finishes the proof of statement $(1)$.

$(2)$ Let $T_1$ be partial tilting and consider the complete tilting cotorsion pair $(^{\perp_1}(T_1^{\perp_1}),T_1^{\perp_1})$. Since $T_1$ is of projective dimension at most one, so are all the modules in $^{\perp_1}(T_1^{\perp_1})$. By \cite[Proposition 3.3]{BET}, there is a filtration $(Y_\alpha\mid \alpha\le\lambda)$ of $T_1$ such that for all $\alpha+1\le\lambda$ the quotient $Y_{\alpha+1}/Y_\alpha$ is countably presented and lies in $^{\perp_1}(T_1^{\perp_1})$. We set $\Scal:=\{Y_{\alpha +1}/Y_\alpha\mid\alpha<\lambda\}$. Using that $\Scal$ is contained in $^{\perp_1}(T_1^{\perp_1})$ and, on the contrary, $T_1\in{^{\perp_1}(\Scal^{\perp_1})}$ by Proposition~\ref{prop cotorsion pairs}, it follows that $\Scal^{\perp_1}=T_1^{\perp_1}$.

Now let $X$ be in $\Scal^\perp$. Since, by construction, $T_1$ is $\Scal$-filtered, $X$ lies in $T_1^\perp$ by the same transfinite induction as in the proof of statement $(1)$. Conversely, take $X$ in $T_1^\perp\subseteq\Scal^{\perp_1}$. It is enough to prove that $X$ is in $Y_\alpha^{\perp_0}$ for all $\alpha\le\lambda$, as then automatically $X \in (Y_{\alpha+1}/Y_\alpha)^{\perp_0}$ for each $\alpha<\lambda$. Consider for $\alpha\le\lambda$ the short exact sequence
$$\xymatrix{0\ar[r] & Y_\alpha\ar[r] & T_1\ar[r] & T_1/Y_\alpha\ar[r] & 0}$$
and apply to it the functor $Hom_A(-,X)$. We obtain the exact sequence
$$\xymatrix{\cdots\ar[r] & Hom_A(T_1,X)\ar[r] & Hom_A(Y_\alpha,X)\ar[r] & Ext_A^1(T_1/Y_\alpha,X)\ar[r] & \cdots}$$
By assumption, we have $Hom_A(T_1,X)=0$ and, moreover, since $T_1/Y_\alpha$ is $\Scal$-filtered and $X\in\Scal^{\perp_1}$, it follows from Proposition~\ref{prop cotorsion pairs} that $Ext_A^1(T_1/Y_\alpha,X)=0$. Hence, we get $Hom_A(Y_\alpha,X)=0$, as wanted. This finishes the proof.
\end{proof}

\begin{remark}\label{open question}
Theorem~\ref{Main 1}(2) is optimal in the sense that we cannot always choose the set $\Scal$ to contain only finitely presented modules of projective dimension at most one. We refer to Example~\ref{example not localisation concrete} for an instance of this phenomenon.
\end{remark}

It is a consequence of the proof of Theorem \ref{Main 1} that in both cases $(1)$ and $(2)$, additionally to $\Scal^{\perp_1}=T_1^{\perp_1}$ and $\Scal^\perp=T_1^\perp$, we also have $\Scal^{\perp_0}\subseteq T_1^{\perp_0}$. The converse inclusion, however, does not hold in general, as shown in the following example.

\begin{example}
Let $A=\mathbb{Z}$ be the ring of integers and consider the abelian group $\mathbb{Z}_{p}=\mathbb{Z}/p\mathbb{Z}$ for a prime $p$. We set $\Scal:=\{\mathbb{Z}_{p}\}$. Following the construction in the proof of Theorem \ref{Main 1}, an associated partial tilting module $T_1$ appears as the middle term of the following approximation sequence
$$\xymatrix{0\ar[r] & \mathbb{Z}_{p}\ar[r] & \mathbb{Z}(p^\infty)\ar[r] & \mathbb{Z}(p^\infty)\ar[r] & 0}$$
where $\mathbb{Z}(p^\infty)=\varinjlim(\frac{1}{p^n}\mathbb{Z})/\mathbb{Z}$ denotes the Pr\"ufer $p$-group. By construction, $\mathbb{Z}(p^\infty)$ is $\Scal$-filtered and, since $\mathbb{Z}(p^\infty)$ is injective, we get $Ext_\mathbb{Z}^1(\mathbb{Z}_{p},\mathbb{Z}(p^\infty))=0$. Using Theorem \ref{Main 1}, it follows that $\mathbb{Z}_{p}^{\perp_1}=\mathbb{Z}(p^\infty)^{\perp_1}$ and $\mathbb{Z}_{p}^{\perp}=\mathbb{Z}(p^\infty)^{\perp}$. Note that the latter subcategory of $Mod(\mathbb{Z})$ describes precisely the modules over the localisation of $\mathbb{Z}$ at the multiplicative set $\{p^n\mid n\geq 0\}$. Moreover, since the module $\mathbb{Z}(p^\infty)$ is $\mathbb{Z}_{p}$-filtered, we get an inclusion $\mathbb{Z}_{p}^{\perp_0}\subseteq\mathbb{Z}(p^\infty)^{\perp_0}$. We claim that this inclusion is strict. It is enough to check that $Hom_\mathbb{Z}(\mathbb{Z}(p^\infty),\mathbb{Z}_{p})=0$. But this follows from the fact that a non-zero map from $\mathbb{Z}(p^\infty)$ to $
 \mathbb{Z}_{p}$ would give rise to an infinite proper subgroup of $\mathbb{Z}(p^\infty)$ which does not exist.
\end{example} 


\section{Tilting ring epimorphisms and universal localisations}\label{section loc}

\subsection{Ring epimorphisms}\label{subsection ring epi}
We call a ring homomorphism a \textbf{ring epimorphism} if it is an epimorphism in the category of rings (with unit). Two ring epimorphisms $f\colon A\ra B$ and $g\colon A\ra C$ are said to be \textbf{equivalent} if there is a (necessarily unique) isomorphism of rings $h\colon B\ra C$ such that $h\circ f=g$. Ring epimorphisms are relevant to study full embeddings of module categories. More precisely, due to \cite{GdP,GL}, there is a bijection between equivalence classes of ring epimorphisms $A\ra B$ and so-called \textbf{bireflective} subcategories of $Mod(A)$ (i.e., full subcategories closed under kernels, cokernels, products and coproducts). The bijection is given by assigning to a ring epimorphism $f\colon A\ra B$ the essential image $\Xcal_B$ of the restriction functor along $f$. That is, $\Xcal_B$ is the full subcategory of all $A$-modules $M$ for which the action $A\to End_\Z(M)$ lifts over $f$ to a $B$-module structure $B\to End_\Z(M)$ (such a lifting, if it exists, is necessarily unique). Moreover, it was shown in \cite[Theorem 4.8]{Sch0} that $Tor_1^A(B,B)=0$ if and only if the associated bireflective subcategory is closed for extensions in $Mod(A)$. We are mainly interested in the following classes of ring epimorphisms.

\subsection{Tilting ring epimorphisms}
Fix a partial tilting $A$-module $T_1$. It follows from \cite{CTT} that the subcategory $T_1^\perp$ of $Mod(A)$ is bireflective (and closed for extensions) and, hence, there is a ring epimorphism $A\ra B$ with $Tor_1^A(B,B)=0$ and such that $\Xcal_B=T_1^{\perp}$. We say that $A\ra B$ is a \textbf{tilting ring epimorphism}. It follows directly from Theorem \ref{Main 1}(2) that tilting ring epimorphisms are of countable type:

\begin{proposition}
Let $A\ra B$ be a tilting ring epimorphism. Then there is a set $\Scal$ of countably presented $A$-modules of projective dimension at most one such that $\Xcal_B=\Scal^\perp$.
\end{proposition}

\begin{example} \label{expl proj}
Let $A$ be a left-noetherian ring and let $A\ra A/I$ be a surjective ring epimorphism with $Tor_1^A(A/I,A/I)=0$. Using $Tor_1^A(A/I,A/I)\cong I/I^2$, it follows that $I$ is idempotent. Since, by assumption, $I$ is a finitely generated left $A$-module, $I$ is the trace ideal of a countably generated projective $A$-module $P$ (see \cite{W}). In particular, we get $\Xcal_{A/I}=P^\perp$ and, hence, $A\ra A/I$ is a tilting ring epimorphism. 
\end{example}

\begin{example} \label{tilting pd1}
Let $A\ra B$ be an injective ring epimorphism with $Tor_1^A(B,B)=0$ and such that the projective dimension of the $A$-module $B$ is at most one. Then, by \cite[Theorem 3.5]{AS}, the $A$-module $B\oplus B/A$ is tilting. Moreover, since $A\ra B$ is injective, $\Xcal_B$ coincides with $B/A^\perp$. Hence, $A\ra B$ is a tilting ring epimorphism.
\end{example}

\subsection{Universal localisations}\label{subsection loc}
The concept goes back to \cite[Theorem 4.1]{Sch0}.

\begin{theorem}\label{uni loc}
Let $\Sigma$ be a set of maps in $proj(A)$. Then there is a ring $A_\Sigma$ and a ring homomorphism $f:A\rightarrow A_\Sigma$ such that
\begin{enumerate}
\item $A_\Sigma \otimes_A \sigma$ is an isomorphism for all $\sigma\in\Sigma$;
\item every ring homomorphism $g:A\rightarrow B$ such that $B\otimes_A \sigma$ is an isomorphism for all $\sigma\in\Sigma$ factors uniquely through $f$.
\end{enumerate}
\end{theorem}

We say that the ring $A_\Sigma$ in Theorem \ref{uni loc} is the \textbf{universal localisation of $A$ at $\Sigma$}. It is well-known that the homomorphism $f\colon A\rightarrow A_\Sigma$ is a ring epimorphism with $Tor_1^A(A_\Sigma,A_\Sigma)=0$ (see \cite[Theorem 4.7]{Sch0}). The essential image of the associated restriction functor will be denoted by $\Xcal_\Sigma$; it consists of all $X\in Mod(A)$ such that $Hom_A(\sigma,X)$ is an isomorphism (\cite[Theorem 4.7]{Sch0}, see also \cite[Lemma 4.4]{BS-smash}).

\begin{example}\label{classical vs universal loc}
Let $A$ be a commutative ring and $\Sigma \subseteq A$ be a multiplicative set. We can interpret each element of $\Sigma$ as a map $\sigma\cdot -\colon A \to A$ in $proj(A)$. Hence, the classical localisation of $A$ at $\Sigma$ is an instance of a universal localisation. In this case, one naturally expresses $\Xcal_\Sigma$ as the intersection of a torsion class, namely $\Dcal_\Sigma = \{ M \in Mod(A) \mid \sigma M = M \}$ of all $\Sigma$-divisible modules, and the torsion-free class $\Fcal_\Sigma$ of all modules for which the action of each $\sigma\in\Sigma$ is injective.

If, moreover, $A$ is local, then the classes of universal localisations and classical localisations coincide. Indeed, every projective $A$-module is then free, and universal localisations boil down to making certain matrices over $A$ universally invertible. However, making a square matrix over a commutative ring invertible is equivalent to making its determinant invertible, and a non-square matrix $M$ cannot be invertible over a non-zero commutative ring $A$, or else $M$ would be invertible over the residue field $\Kbb(\mathfrak{p})$ at any fixed prime ideal $\mathfrak{p}\subseteq A$, which is absurd.

For non-local commutative rings, however, the class of universal localisations is in general broader than that of classical localisations. For Dedekind domains this is determined by the ideal class group, see~\cite[Proposition 6.8 and Corollary 6.17]{AS}.
\end{example}

An analogous pair of torsion pairs involving classes of torsion and divisible modules as in the above example can be associated with any set $\Sigma$ of maps in $proj(A)$ for an arbitrary ring $A$. 
One of the torsion pairs, $(\Dcal_\Sigma,\Rcal_\Sigma)$, is given by 
$$\Dcal_\Sigma:=\{X\in Mod(A) \mid Hom_A(\sigma,X)\text{\, is surjective\,} \forall\sigma\in\Sigma\}.$$
The modules in $\Dcal_\Sigma$ (respectively, $\Rcal_\Sigma$) will be called \textbf{$\Sigma$-divisible} (respectively, \textbf{$\Sigma$-reduced}).
In the second torsion pair, $(\Tcal_\Sigma,\Fcal_\Sigma)$, the torsion class is generated by all the cokernels of maps in $\Sigma$. Hence, we have
$$\Fcal_\Sigma:=\{X\in Mod(A) \mid Hom_A(\sigma,X)\text{\, is injective\,} \forall\sigma\in\Sigma\}.$$
The modules in $\Fcal_\Sigma$ (respectively, $\Tcal_\Sigma$) will be called \textbf{$\Sigma$-torsion-free} (respectively, \textbf{$\Sigma$-torsion}).
In particular, it follows that $\Xcal_\Sigma=\Dcal_\Sigma\cap\Fcal_\Sigma.$
Similar torsion and torsion-free classes have appeared before in the context of universal localisations (see, for example, \cite{Sch1} and \cite[Chapter 4]{AS}), but it seems that a general discussion of their properties and uses is still missing. The next easy example shows that the above torsion pairs depend on the set $\Sigma$ and not only on the universal localisation $A_\Sigma$ of~$A$.

\begin{example}
Let $P$ be a non-zero finitely generated projective $A$-module and consider the sets $\Sigma=\{0\ra P\}$ and $\Sigma'=\{P\ra 0\}$. Clearly, we have $A_\Sigma=A_{\Sigma'}$. The torsion class $\Dcal_\Sigma$ equals $Mod(A)$. But $\Dcal_{\Sigma'}$, on the other hand, is given by $\Xcal_\Sigma\subsetneq Mod(A)$. Similarly, we have $\Fcal_\Sigma\neq\Fcal_{\Sigma'}$.
\end{example}

We can pass information from $\Sigma$ to the class of $\Sigma$-divisible modules. 

\begin{lemma}\label{mono}
Let $\Sigma$ be a set of maps in $proj(A)$. The following are equivalent.
\begin{enumerate}
\item The set $\Sigma$ consists of monomorphisms;
\item $\Dcal_\Sigma$ is a tilting class.
\end{enumerate}
Moreover, every tilting class in $Mod(A)$ arises in this way.
\end{lemma}

\begin{proof}
Suppose that (1) holds and consider the set $\Scal=\{Coker(\sigma)\mid\sigma\in\Sigma\}$ of finitely presented modules of projective dimension at most one. It follows that $\Dcal_\Sigma=\Scal^{\perp_1}$, which is a tilting class by Theorem~\ref{tilting finite type} and the discussion below it. Conversely, by using again Theorem~\ref{tilting finite type}, every tilting class in $Mod(A)$ is of the form $\Scal^{\perp_1}$ for a set $\Scal$ of finitely presented modules of projective dimension at most one. Therefore, it can be written as $\Dcal_\Sigma$ for a set $\Sigma$ in $proj(A)$ of monomorphic presentations of the modules in $\Scal$.

Suppose that (2) holds. We need to show that all the maps in $\Sigma$ are monomorphic.
First, note that an $A$-module $X$ is in $\Dcal_\Sigma$ if and only if for all $\sigma\colon P\ra Q$ in $\Sigma$ every given map $P\ra X$ factors through $\sigma$. Now take $\sigma\colon P\ra Q$ in $\Sigma$ and let $C$ be an injective cogenerator of $Mod(A)$ together with a monomorphism $\phi\colon P\ra C^{I}$ for some set $I$.
Since $C^I$ is an injective $A$-module and, by assumption, $\Dcal_\Sigma$ is a tilting class, we have $C^I\in\Dcal_\Sigma$. Consequently, the map $\phi$ must factor through $\sigma$, forcing also the map $\sigma$ to be a monomorphism.
\end{proof}

If the equivalent conditions in Lemma \ref{mono} are fulfilled, the universal localisation of $A$ at $\Sigma$ turns out to be easier to describe (see, for example, \cite{Sch1}). The following result is a consequence of Theorem \ref{Main 1}.

\begin{proposition}\label{mono tilting epi}
Let $\Sigma$ be a set of monomorphic maps in $proj(A)$. Then the universal localisation $A\ra A_\Sigma$ is a tilting ring epimorphism.
\end{proposition}

\begin{proof}
As observed in the proof of Lemma \ref{mono} above, the class $\Dcal_\Sigma$ equals $\Scal^{\perp_1}$ for $\Scal=\{Coker(\sigma)\mid\sigma\in\Sigma\}$ and, moreover, we have $\Xcal_\Sigma=\Dcal_\Sigma\cap\Fcal_\Sigma=\Scal^\perp$. Now the claim follows from Theorem \ref{Main 1}(1).
\end{proof}

In general, however, not every universal localisation of $A$ is given by localising at a set of monomorphic maps and not every universal localisation is a tilting ring epimorphism. In order to better understand the general situation, we will pass in the forthcoming sections to the notion of silting module. 

Alternatively, we can use the following trick of torsion-reduction.
Note that for any set $\Sigma$ in $proj(A)$ the universal localisation of $A$ at $\Sigma$ coincides with the universal localisation of $A$ at $\Sigma^*:=\{Hom_A(\sigma,A) \mid \sigma\in\Sigma\}$ in $proj(A^{op})$ (Theorem \ref{uni loc} can also be stated for right $A$-modules!). Let $A_{TF}$ be the ring we obtain by factoring out the $\Tcal_\Sigma$-torsion part of $A$. The associated surjective ring homomorphism is denoted by $\pi:A\ra A_{TF}$. By construction, there is a unique $A$-module map $g:A_{TF}\ra A_\Sigma$ yielding a commutative diagram of ring epimorphisms
$$\xymatrix{A\ar[rr]^f\ar[dr]_\pi &  & A_\Sigma\\ & A_{TF}.\ar[ur]_g & }$$

\begin{proposition} \label{silting to tilting TF}
The map $g:A_{TF}\ra A_\Sigma$ is the universal localisation of $A_{TF}$ at $\Sigma_{TF}:=\{A_{TF}\otimes_A \sigma \mid \sigma\in\Sigma\}$. Moreover, $g$ is given by localising at the set of injective morphisms in $\Sigma^*_{TF}$.
\end{proposition}

\begin{proof}
First of all, $A_\Sigma$ is $\Sigma_{TF}$-invertible, since for all $\sigma\in\Sigma$
$$A_\Sigma\otimes_{A_{TF}}(A_{TF}\otimes_A\sigma)\cong A_\Sigma\otimes_A\sigma$$
is an isomorphism by assumption on $A_\Sigma$. We have to check the universal property for $g$. Let $\psi:A_{TF}\ra S$ be a $\Sigma_{TF}$-invertible ring homomorphism. Consequently,
$$S\otimes_A\sigma\cong S\otimes_{A_{TF}}(A_{TF}\otimes_A\sigma)$$
is an isomorphism yielding, by the universal property of $A_\Sigma$, a unique ring homomorphism $h:A_\Sigma\ra S$ such that $\psi\circ\pi=h\circ f$. Hence, we get that $\psi\circ\pi=h\circ g\circ\pi$. Using the surjectivity of $\pi$, we obtain the wanted factorisation $\psi=h\circ g$. It follows that $g$ is the universal localisation of $A_{TF}$ at $\Sigma_{TF}$. Moreover, by the construction of $A_{TF}$, we know that $Hom_A(\sigma,A_{TF})$ is injective and, thus, using adjunction, also
$Hom_{A_{TF}}(A_{TF}\otimes_A\sigma,A_{TF})$
must be injective. This finishes the proof.
\end{proof}

\begin{remark}\label{remark torsion reduction}
We can iterate the process described above. To do so, we first need to define - similar to the situation for left $A$-modules - the torsion pairs $(\Dcal^{op},\Rcal^{op})$ and $(\Tcal^{op},\Fcal^{op})$ in $Mod(A^{op})$ with respect to a given set $\Sigma^{op}$ in $proj(A^{op})$. Now we can reduce $A$ transfinitely by factoring out, step by step, the $\Tcal$-torsion and the $\Tcal^{op}$-torsion part. As a direct limit, we obtain a ring $A_{TF}$ that is torsion-free from both sides, meaning, with respect to $\Fcal$ and $\Fcal^{op}$. Note that the reduction process is finite if the ring $A$ is noetherian. Again, we get a commutative diagram of ring epimorphisms as above
where $g$ is the universal localisation of $A_{TF}$ at $\Sigma_{TF}$. But now all the maps in $\Sigma_{TF}^*$ \textbf{and} in $\Sigma_{TF}$ are injective or, equivalently, all cokernels of the maps in $\Sigma_{TF}$ are bound. The latter is a necessary (see \cite[Remark 4.4]{AS}) but not a sufficient condition for the localisation $g$ to be injective. An example of a universal localisation at a set of bound modules that is not injective can be found in \cite[Example 4.3]{AMV2} by localising with respect to the bound module $I_5\oplus S_3$.
\end{remark}

A major problem with using Proposition~\ref{silting to tilting TF} and Remark \ref{remark torsion reduction} in practice is that, in general, $Tor_1^A(A_{TF},A_{TF})$ does not vanish. In particular, $\pi$ may not be a universal localisation. Thus, there is very little one can say about the relation of the homological properties of $A$ and those of $A_{TF}$. 


\section{The morphism category}\label{section mor}
In this section, we provide the necessary setup to discuss universal localisations via the notion of silting module which gives an arguably more practical approach to treat universal localisations at non-injective maps than Proposition~\ref{silting to tilting TF} (see also the applications in \cite{MS} supporting this claim).

For a given ring $A$, we are interested in the morphism category $Mor(A)$ whose objects $Z_g$ are $A$-module maps $g\colon M\ra N$ and whose morphisms are given by commutative squares of the form

$$\xymatrix{M\ar[r]^g\ar[d] & N\ar[d]\\ M'\ar[r]^{g'}& N'.}$$

Note that the category $Mor(A)$ is equivalent to the category of left modules over the lower triangular matrix ring $T_2(A)$ (see \cite[Chapter I.4]{A} and \cite[Chapter III.2]{ARS}).
Moreover, every object in $Mor(A)$ can be viewed as a two-term complex of $A$-modules concentrated in homological degrees $1$ and $0$.
We are particularly interested in the following two full and extension-closed subcategories of $Mor(A)$

$$\Lcal:=Mor(proj(A))=\{Z_\sigma \mid \sigma\in proj(A)\}$$
$$\Bcal\Lcal:=Mor(Proj(A))=\{Z_\sigma \mid \sigma\in Proj(A)\}.$$

Both $\Lcal$ and $\Bcal\Lcal$ are naturally equipped with the structure of an exact category induced by $Mor(A)$ where conflations are defined as degreewise exact sequences.
Furthermore, it is not hard to check that $\Bcal\Lcal$ is closed for filtrations in $Mor(A)$. In fact, $\Bcal\Lcal$ appears on the left hand side of the cotorsion pair generated by $Z_{(A\ra 0)}$.
The following lemma will be crucial in our context.

\begin{lemma}\label{lem hered}
The categories $\Lcal$ and $\Bcal\Lcal$ are hereditary with enough projectives and injectives. Moreover, for all objects $Z_\sigma$ in $\Bcal\Lcal$ and $Z_g$ in $Mor(A)$, we have
$$Ext_{Mor(A)}^1(Z_\sigma,Z_g)\cong Hom_{\Kcal^b(A)}(Z_\sigma,Z_g[1]).$$
\end{lemma}

\begin{proof}
Note that for $P$ in $Proj(A)$ the objects $Z_{id_P}$ and $Z_{(0\ra P)}$ are projective in $Mor(A)$ and, thus, also in $\Bcal\Lcal$. Dually, the objects $Z_{id_P}$ and $Z_{(P\ra 0)}$ are injective in $\Bcal\Lcal$ (but usually not in $Mor(A)$).

Now take some $Z_\sigma$ in $\Bcal\Lcal$ given by the map $\sigma\colon P\ra Q$. A projective resolution 
$$\xymatrix{0\ar[r] & P_1(Z_\sigma)\ar[r] & P_0(Z_\sigma)\ar[r] & Z_\sigma\ar[r] & 0}$$
of $Z_\sigma$ in $\Bcal\Lcal$ (equivalently, in $Mor(A)$) is given by the following commutative diagram of $A$-modules
$$\xymatrix{0\ar[r] & 0\ar[r]\ar[d] & P\ar[d]^\oplus\ar[r]^{id} & P\ar[d]^\sigma\ar[r] & 0\\ 0\ar[r] & P\ar[r]^{\!\!\!\!\Tiny{\left(\begin{array}{c}-\sigma\\id\end{array}\right)}} & Q\oplus P\ar[r]^{\,\,\,\,(id\,\,\, \sigma)} & Q\ar[r] & 0}$$
Given $g\colon M\ra N$, an element of $Ext_{Mor(A)}^1(Z_\sigma,Z_g)$ is represented by an element of $Hom_{Mor(A)}(P_1(Z_\sigma),Z_g)$, so by a commutative square of $A$-modules:
$$\xymatrix{0\ar[r]\ar[d]& M\ar[d]^g \\ P\ar[r]^h & N}$$
Similarly, an element of $Hom_{\Kcal^b(A)}(Z_\sigma,Z_g[1])$ is represented by a chain complex morphism:
$$\xymatrix{& P\ar[d]^h\ar[r]^\sigma & Q\\ M\ar[r]^{-g} & N &}$$
Both cases amount to specifying a morphism $h\colon P \to N$ in $Mod(A)$.

Now the corresponding element of $Ext_{Mor(A)}^1(Z_\sigma,Z_g)$ vanishes if and only if the map $P_1(Z_\sigma) \to Z_g$ factors through the inclusion $P_1(Z_\sigma) \to P_0(Z_\sigma)$ if and only if there are maps $u\colon P \to M$ and $v\colon Q \to N$ in $Mod(A)$ such that the lower row of
$$\xymatrix{0\ar[r]\ar[d]& P\ar[r]^u\ar[d]^\oplus& M\ar[d]^g \\ P\ar[r]^-{\Tiny{\left(\begin{array}{c}-\sigma\\id\end{array}\right)}}& Q \oplus P\ar[r]^-{(v\,\,\, gu)} & N}$$
composes to $h$. This is further equivalent to $h=gu-v\sigma$, or, in other words, the chain complex morphism above being null-homotopic.
\end{proof}

\begin{remark}\label{mapping cone}
The isomorphism $Hom_{\Kcal^b(A)}(Z_\sigma,Z_g[1]) \ra Ext_{Mor(A)}^1(Z_\sigma,Z_g)$ can also be made rather explicit via the mapping cone construction. Given a chain complex morphism $\tilde h\colon Z_\sigma \to Z_g[1]$, the mapping cone of $\tilde h$ is by the very construction a part of an exact sequence of chain complexes
$$\xymatrix{0 \ar[r] & Z_g \ar[r] & cone(\tilde h) \ar[r] & Z_\sigma \ar[r] & 0}$$
which can be interpreted as a short exact sequence in $Mor(A)$. This yields an element of $Ext_{Mor(A)}^1(Z_\sigma,Z_g)$. As this fact is only supplementary with respect to the discussion below, we leave the details to the reader.
\end{remark}

\begin{remark}\label{rem proj dim}
The previous lemma shows that all objects in $\Bcal\Lcal$ are of projective dimension at most one when regarded as objects in $Mor(A)$. However, the subcategory $\Bcal\Lcal$ is, in general, not determined by this property. For example, if $X \in Mod(A)$ is non-projective of projective dimension one, then $Z_{id_X}$ has projective dimension one in $Mor(A)$, but $Z_{id_X} \not\in \Bcal\Lcal$. 
\end{remark}

We will be interested in translating information from $Mor(A)$ to $Mod(A)$ and backwards. For a set $\Sigma$ of objects in $\Bcal\Lcal$ we consider the full subcategory $\Dcal_\Sigma$ of $Mod(A)$ given by all modules $X$ for which the map $Hom_A(\sigma,X)$ is surjective for all $Z_\sigma\in\Sigma$. If the set $\Sigma$ only contains a single object $Z_\sigma$, we sometimes write $\Dcal_\sigma$ instead of $\Dcal_\Sigma$. According to the definition in the previous section (where the set $\Sigma$ was chosen from $\Lcal$), the $A$-modules in $\Dcal_\Sigma$ are called \textbf{$\Sigma$-divisible}. Note that $\Dcal_\Sigma$ is always closed for extensions and quotients in $Mod(A)$. In other words, $\Dcal_\Sigma$ is a torsion class whenever it is closed for coproducts (for example, in case the set $\Sigma$ is chosen from $\Lcal$). The following lemma is inspired by \cite[Lemma 3.4]{AIR}.

\begin{lemma}\label{translation}
Let $\Sigma$ be a set of objects in $\Bcal\Lcal$. Then the full subcategory $\Sigma^{\perp_1}$ of $Mor(A)$ is given by $\{Z_g\in Mor(A)\mid Coker(g)\in\Dcal_\Sigma\}$.
\end{lemma}

\begin{proof}
By Lemma \ref{lem hered}, an object $Z_g$ in $Mor(A)$ belongs to $\Sigma^{\perp_1}$ if and only if $Hom_{\Kcal^b(A)}(Z_\sigma,Z_g[1])=0$ for all $Z_\sigma\in\Sigma$. Now the claim follows from the proof of \cite[Lemma 3.4]{AIR}. To see this, note that the arguments used there are of pure homological nature and, thus, they can be applied as long as $\Sigma$ only contains morphisms between (not necessarily finitely generated) projective $A$-modules.
\end{proof}


\section{Silting modules and universal localisations}\label{section silting}

\subsection{Silting modules}
Silting modules were introduced as the module-theoretic counterpart of two-term silting complexes (see \cite{AMV}).

\begin{definition}
An $A$-module $T$ is called 
\begin{itemize}
\item \textbf{partial silting}, if there is a projective presentation $\omega$ of $T$ such that $\Dcal_\omega$ is a torsion class and $T\in\Dcal_\omega$.
\item \textbf{silting}, if there is a projective presentation $\omega$ of $T$ such that $\Dcal_\omega=Gen(T)$.
\end{itemize}
We say that $T$ is (partial) silting \textbf{with respect to $\omega$}.
\end{definition}

It is not hard to check that (partial) tilting modules are always (partial) silting. Moreover, in the context of finite dimensional algebras over an algebraically closed field, the finitely presented silting modules correspond precisely to the support $\tau$-tilting modules introduced in \cite{AIR}. 
Now we can also show that (partial) silting modules can be viewed as (partial) tilting objects in the morphism category.

\begin{lemma} \label{lem silting to tilting}
Let $T$ be an $A$-module with a projective presentation $\omega$. Then the following holds.
\begin{enumerate}
\item $T$ is partial silting with respect to $\omega$ if and only if $Z_\omega$ is partial tilting in $Mor(A)$.
\item $T$ is silting with respect to $\omega$ if and only if $Z_\omega \oplus Z_{id_A}$ is tilting in $Mor(A)$.
\end{enumerate}
\end{lemma}

\begin{proof}
(1) By Lemma~\ref{translation}, $T\in \Dcal_\omega$ if and only if $Z_\omega\in Z_\omega^{\perp_1}$ and, moreover, $\Dcal_\omega$ is closed under coproducts in $Mod(A)$ if and only if $Z_\omega^{\perp_1}$ is closed under coproducts in $Mor(A)$. Hence, $\Dcal_\omega$ is a torsion class if and only if so is $Z_\omega^{\perp_1}$. This proves (1).

(2) Suppose first that $Z := Z_\omega \oplus Z_{id_A}$ is tilting, that is $Z^{\perp_1} = Gen(Z)$. Given any morphism $g\colon X \to Y$ in $Mod(A)$, we have canonical isomorphisms
$$ Hom_{Mor(A)}(Z_{id_A},Z_g) \cong Hom_A(A,X) \cong X. $$
In particular, we get $Hom_{Mor(A)}(Z_{id_A},Z_{(0\to Y)}) = 0$, and $Z_{(0\to Y)} \in Gen(Z)$ if and only if $Z_{(0\to Y)} \in Gen(Z_\omega)$ if and only if $Y \in Gen(T)$. On the other hand, by Lemma \ref{translation}, $Z_{(0\to Y)} \in Z_\omega^{\perp_1}$ if and only if $Y \in \Dcal_\omega$. Thus, $T$ is silting with respect to $\omega$.

Suppose conversely that $T$ is silting, so that $Z_\omega$ is partial tilting by (1). As $Z_{id_A}$ is projective injective in $\Bcal\Lcal$ by the proof of Lemma~\ref{lem hered}, also $Z := Z_\omega \oplus Z_{id_A}$ is partial tilting. It remains to show that $Z^{\perp_1} \subseteq Gen(Z)$. To this end, let $g\colon X \to Y$ be a morphism in $Mod(A)$ such that $Z_g \in Z^{\perp_1}$, that is $Coker(g) \in \Dcal_\omega$. Hence, there is a surjection $p_0\colon T^{(I)} \to Coker(g)$ which we can lift to a map $p\colon Z_\omega^{(I)} \to Z_g$ (we recover $p_0$ from $p$ by passing to cokernels). Furthermore, we can take a surjection $q_0\colon A^{(J)} \to X$ which canonically extends to a map $q\colon Z_{id_A} \to Z_g$. A simple diagram chase reveals that $\big({^p_q}\big)\colon Z_\omega^{(I)} \oplus Z_{id_A}^{(J)} \to Z_g$ is surjective. Thus, $Z_g \in Gen(Z)$.
\end{proof}

\subsection{Silting classes}
We are particularly interested in the torsion class $\Dcal_\omega$ associated with a partial silting module $T_1$. Since, by \cite[Theorem 3.12]{AMV}, the module $T_1$ can be completed to a silting module $T=T_0\oplus T_1$ with $Gen(T)=\Dcal_\omega$, we call $\Dcal_\omega$ a  \textbf{silting class}. Note that silting classes are always \textbf{definable}, i.e., they are closed for direct limits, products and pure submodules in $Mod(A)$ (see \cite[Corollary 3.5]{AMV}). Here, we will give a finite-type characterisation of silting classes which generalises Theorem~\ref{tilting finite type} (also compare to Lemma \ref{mono}).

\begin{theorem}\label{main silting classes}
Let $\Dcal$ be a full subcategory of $Mod(A)$. Then the following are equivalent.
\begin{enumerate}
\item $\Dcal$ is a silting class;
\item $\Dcal=\Dcal_\Sigma\,$ for a set $\Sigma$ of objects in $\Lcal$;
\item $\Dcal=\Dcal_\Sigma\,$ for a set $\Sigma$ of objects in $\Bcal\Lcal$ and $\Dcal$ is closed for coproducts.
\end{enumerate}
\end{theorem}

\begin{proof}
$(1)\Rightarrow (2)$ Assume that $\Dcal=\Dcal_\omega$ for a partial silting module $T_1$ with respect to $\omega$. Then, $Z_\omega$ is partial tilting by Lemma~\ref{lem silting to tilting}(1) and, by Theorem \ref{tilting finite type}, there is a set of finitely presented objects $\Sigma$ in $Mor(A)$ such that $\Sigma^{\perp_1}=Z_\omega^{\perp_1}$. Moreover, following the construction in Section \ref{section tilting}, $\Sigma$ is contained in ${^{\perp_1}(Z_\omega^{\perp_1})} \subseteq \Bcal\Lcal$. However, the finitely presented objects of $Mor(A)$ which are contained in $\Bcal\Lcal$ are precisely those in $\Lcal$, and it follows from Lemma \ref{translation} that $\Dcal=\Dcal_\Sigma$.

The implication $(2)\Rightarrow (3)$ is clear. We are left to prove $(3)\Rightarrow (1)$. Assume that $\Dcal=\Dcal_\Sigma\,$ for a set $\Sigma$ of objects in $\Bcal\Lcal$ and consider the cotorsion pair $(^{\perp_1}(\Sigma^{\perp_1}),\Sigma^{\perp_1})$ in $Mor(A)$. By assumption, $\Dcal$ is closed for coproducts in $Mod(A)$ and, thus, so is $\Sigma^{\perp_1}$ in $Mor(A)$ by Lemma~\ref{translation}. Moreover, every object in $^{\perp_1}(\Sigma^{\perp_1})\subseteq\Bcal\Lcal$ has projective dimension at most one (see Remark \ref{rem proj dim}). Hence, $\Sigma^{\perp_1}$ is a tilting class in $Mor(A)$ by \cite[Theorem 4.1]{AC} (see also the discussion of Theorem \ref{tilting finite type} in Section \ref{section tilting}).
Let $Z_\omega$ be an associated tilting module with $Z_\omega^{\perp_1}=\Sigma^{\perp_1}$. Note that, by construction, $Z_\omega$ belongs to ${^{\perp_1}(\Sigma^{\perp_1})} \subseteq \Bcal\Lcal$. Now it follows from Lemma~\ref{translation} and Lemma~\ref{lem silting to tilting}(2) that $Coker(\omega)$ is a silting $A$-module with respect to $\omega$ and $\Dcal_\omega=\Dcal$. 
\end{proof}

\begin{remark}
The implication $(1)\Rightarrow (2)$ in Theorem \ref{main silting classes} also follows from recent independent work by Angeleri H\"ugel and Hrbek (see \cite[Theorem 2.3]{AH}). Moreover, the authors show that over a noetherian ring definable torsion classes coincide precisely with silting classes (\cite[Corollary 3.7]{AH}) and that such a statement is no longer true over arbitrary rings (see \cite[Example 5.4]{AH}).
\end{remark}

\subsection{Silting ring epimorphisms}
In \cite{AMV2}, it was shown that partial silting modules give rise to ring epimorphisms. Let us briefly recall this construction. Fix a partial silting module $T_1$ with respect to a presentation $\omega$ in $Proj(A)$. It follows that the intersection 
$$\Dcal_\omega\cap T_1^{\perp_0}=\{X\in Mod(A)\mid Hom_A(\omega,X) \text{ is an isomorphism} \}$$ 
is a bireflective and extension-closed subcategory of $Mod(A)$ (compare to the similar construction for universal localisations in \S\ref{subsection loc}). In particular, there is an associated ring epimorphism $A\ra B$ with $Tor_1^A(B,B)=0$ and such that $\Xcal_B=\Dcal_\omega\cap T_1^{\perp_0}$ (see \S\ref{subsection ring epi}). We say that $A\ra B$ is a \textbf{silting ring epimorphism}. Clearly, tilting ring epimorphisms provide examples of silting ring epimorphisms.

\begin{example} \label{silting pd1}
Let $f\colon A\ra B$ be a ring epimorphism with $Tor_1^A(B,B)=0$ and such that the projective dimension of the $A$-module $B$ is at most one. Such a situation arises on various occasions, e.g.~if $A$ is hereditary, or if $A$ is commutative and $B$ is a localisation of $A$ at a countable multiplicative set. We will show that $B\oplus B/f(A)$ is a silting $A$-module which turns $f$ into a silting ring epimorphism. Note that this generalises Example~\ref{tilting pd1}.

Let $0 \to P_1 \overset{p}\to P_0 \to B \to 0$ be a projective presentation of $B$ as an $A$-module and consider a lift of $f$ to a map $\tilde{f}\colon Z_{(0 \to A)} \to Z_p$ in $Mor(A)$ (so that we recover $f$ by passing to cokernels). A pushout of $\tilde{f}$ along the obvious inclusion $Z_{(0 \to A)} \to Z_{id_A}$ induces a short exact sequence in $Mor(A)$ whose last term is the mapping cone of $\tilde{f}$ (when we view $\tilde{f}$ as a map of two-term complexes):
$$
\vcenter{
\xymatrix{
0 \ar[r] & 0 \ar[r] \ar[d] & A \oplus P_1 \ar[r] \ar[d]^{id_A \oplus p} & A \oplus P_1 \ar[r] \ar[d]^\omega & 0\phantom{.} \\ 
0 \ar[r] & A \ar[r]        & A \oplus P_0 \ar[r]                        & P_0 \ar[r]                        & 0.
}
}
\eqno{(*)}
$$
Clearly, $\omega$ is a projective presentation of $B/f(A)$.

We claim that $B \oplus B/f(A)$ is a silting $A$-module with respect to $p\oplus\omega$. First, note that $\Dcal_{p \oplus \omega} = \Dcal_\omega$ thanks to~$(*)$. Thus, we need to show that $Gen(B)=\Dcal_\omega$. Observe that $\textbf{R}Hom_A(f,B^{(I)})$ is an isomorphism for every set $I$ since $Hom_A(f,B^{(I)})$ is an isomorphism and $Ext^1_A(B,B^{(I)}) = Ext^1_B(B,B^{(I)}) = 0$. Since $\omega$ is quasi-isomorphic to the mapping cone of $f$, it follows that $Hom_A(\omega,B^{(I)})$ is an isomorphism and, in particular, $Gen(B) \subseteq \Dcal_\omega$. Conversely, given $X \in \Dcal_\omega$, we apply $Hom_{Mor(A)}(-,Z_{(0\to X)})$ to~$(*)$ and obtain an exact sequence
$$\xymatrix{ Hom_{Mor(A)}(Z_{id_A}\oplus Z_p, Z_{(0\to X)}) \ar[r] & Hom_{Mor(A)}(Z_{(0\to A)}, Z_{(0\to X)}) \ar[r] & 0.} $$
Since, by construction, the latter epimorphism identifies with
$$ Hom_A(f,X)\colon Hom(B,X) \to Hom_A(A,X) \cong X, $$
we have $X \in Gen(B)$. This proves the claim.

Finally, $X \in Mod(A)$ is in the essential image of the restriction along $f$ if and only if $Hom_A(f,X)$ is an isomorphism if and only if $\textbf{R}Hom_A(f,X)$ is an isomorphism if and only if $Hom_A(\omega,X)$ is an isomorphism. Hence $\Xcal_B = \Dcal_\omega \cap (B/f(A))^{\perp_0}$.
\end{example}

Now we can state and prove our main results. We first show that silting ring epimorphisms are of countable type.

\begin{theorem}\label{silting countable type}
Let $A\ra B$ be a silting ring epimorphism. Then there is a set $\Sigma$ of countably generated objects in $\Bcal\Lcal$ such that $\Xcal_B=\Dcal_\Sigma\cap Coker(\Sigma)^{\perp_0}$ where $Coker(\Sigma)=\{Coker(\sigma)\mid Z_\sigma\in\Sigma\}$.
\end{theorem}

\begin{proof}
Let $T_1$ be a partial silting module with respect to a presentation $\omega$ such that $\Xcal_B=\Dcal_\omega\cap T_1^{\perp_0}$. The object $Z_\omega$ is partial tilting in $Mor(A)$ by Lemma~\ref{lem silting to tilting}(1) and, thus, by Theorem \ref{Main 1}(2), there is a set $\Sigma$ of countably generated objects in $Mor(A)$ such that $Z_\omega^\perp=\Sigma^\perp$. Note that, by construction, $\Sigma \subseteq {^{\perp_1}(Z_\omega^{\perp_1})} \subseteq \Bcal\Lcal$. We need to show that $\Dcal_\omega\cap T_1^{\perp_0}=\Dcal_\Sigma\cap Coker(\Sigma)^{\perp_0}$. Take $X$ in $Mod(A)$ and consider the object $Z_X := Z_{(0\ra X)}$ in $Mor(A)$. If $\sigma\in\Bcal\Lcal$, we have $Z_X \in Z_\sigma^{\perp_0}$ if and only if $Hom_A(\sigma,X)$ is injective, and by Lemma \ref{translation} also that $Z_X \in Z_\sigma^{\perp_1}$ if and only if $Hom_A(\sigma,X)$ is surjective.
In particular, $X$ belongs to $\Dcal_\omega\cap T_1^{\perp_0}$ if and only if $Z_X \in Z_\omega^\perp$, and similarly $X$ belongs to $\Dcal_\Sigma\cap Coker(\Sigma)^{\perp_0}$ if and only if $Z_X \in \Sigma^\perp$. It follows, $\Dcal_\omega\cap T_1^{\perp_0}=\Dcal_\Sigma\cap Coker(\Sigma)^{\perp_0}$ because $Z_\omega^\perp=\Sigma^\perp$.
\end{proof}

If we can choose the set $\Sigma$ in Theorem \ref{silting countable type} to be contained in $\Lcal\subseteq Mor(A)$, then the silting ring epimorphism turns out to be the universal localisation of $A$ at $\Sigma$. In general, however, silting ring epimorphisms will not be universal localisations. An example will be given in \S\ref{example not localisation}, but since the computation is more involved, we have postponed it to the very end of the paper.
On the other hand, it turns out that universal localisations are \emph{always} silting ring epimorphisms and, hence, are controlled by (partial) silting modules.

\begin{theorem}\label{Main 2}
Every universal localisation is a silting ring epimorphism.
\end{theorem}

\begin{proof}
Let $\Sigma$ be a set of objects in $\Lcal\subseteq Mor(A)$ and let $A\ra A_\Sigma$ be the associated universal localisation. We need to show that there is a partial silting $A$-module $T_1$ with respect to a presentation $\omega$ such that $\Xcal_\Sigma=\Dcal_\omega\cap T_1^{\perp_0}$. By Theorem \ref{Main 1}(1), there is a partial tilting module $Z_\omega$ in $Mor(A)$ such that $Z_\omega^\perp=\Sigma^\perp$. By construction, $Z_\omega$ must belong to $\Bcal\Lcal$, and by Lemma~\ref{lem silting to tilting}, $T_1:=Coker(\omega)$ is a partial silting $A$-module with respect to $\omega$. Hence, it suffices to check that $\Xcal_\Sigma=\Dcal_\Sigma\cap Coker(\Sigma)^{\perp_0}$ coincides with $\Dcal_\omega\cap T_1^{\perp_0}$. But this follows from the equality $Z_\omega^\perp=\Sigma^\perp$ as shown in the proof of Theorem \ref{silting countable type}.
\end{proof}

\begin{corollary}
Let $A\ra A_\Sigma$ be a universal localisation. Then there is a silting module $T$ such that $A_\Sigma$ is isomorphic (as rings) to $End_A^{op}(T)/I$ for some two-sided idempotent ideal $I$ of $End_A^{op}(T)$.
\end{corollary}

\begin{proof}
By Theorem \ref{Main 2}, every universal localisation arises from a partial silting module and, thus, we can apply \cite[Theorem 3.5]{AMV2}.
\end{proof}

\begin{remark}
Let $A\ra A_\Sigma$ be a universal localisation and $T_1$ a partial silting module with respect to $\omega$ such that $\Xcal_\Sigma=\Dcal_\omega\cap T_1^{\perp_0}$.
Then, in general, the object $Z_\omega$ does not belong to $\Lcal\subseteq Mor(A)$. But if it does, the situation becomes significantly nicer. Instead of localising at the given set $\Sigma$, it is enough to consider the localisation at $\{\omega\}$. In this case, $A_\Sigma$ will always be finitely generated when seen as an $A$-module. Examples of such localisations can be found in \cite{MS}.
\end{remark}

\begin{remark}
Given a set $\Sigma$ of objects in $\Lcal\subseteq Mor(A)$, we define the universal localisation of $T_2(A)$ (the lower triangular matrix ring in $A$) at $\Sigma$ by localising with respect to a set of monomorphic projective presentations of the objects in $\Sigma$. In particular, we are in the setting of Lemma \ref{mono} and the universal localisation $T_2(A)\ra T_2(A)_\Sigma$ is a tilting ring epimorphism (see Proposition \ref{mono tilting epi}).

The situation becomes even nicer if we assume that $Z_{id_A}$ belongs to $\Sigma$ (note that adding $Z_{id_A}$ to $\Sigma$ does not affect the universal localisation $A\ra A_\Sigma$). Since $Z_{id_A}$ is projective in $Mor(A) \simeq Mod(T_2(A))$, we can invoke Example~\ref{expl proj}. The trace ideal of $Z_{id_A}$ in $T_2(A)$ is generated by $e = \big(\begin{smallmatrix}1 & 0 \\ 0 & 0\end{smallmatrix}\big)$ and, hence, the universal localisation $T_2(A) \to T_2(A)_\Sigma$ factors as
$$ T_2(A) \to T_2(A)/(e) \cong A \to T_2(A)_\Sigma. $$
The left hand side map is the universal localisation of $T_2(A)$ at the set $\{0 \to Z_{id_A}\}$, and one can check that the map $A \to T_2(A)_\Sigma$ identifies with the usual universal localisation $A \to A_\Sigma$.
The philosophy is that the morphism category does not only serve to reduce silting to tilting, but also to replace general universal localisations by localisations at sets of monomorphic maps. 
\end{remark}

\subsection{A counterexample}\label{example not localisation}

We conclude with the promised example of a silting ring epimorphism which is not a universal localisation.
It is essentially taken from~\cite[Remarks 2.1]{HS} and it was communicated to us by Joe Chuang and Jorge Vit\'{o}ria.

First, we recollect basic facts from commutative algebra about local cohomology. We refer to~\cite{BrSh} for more details, and also to~\cite{Hu} for a brief but efficient introduction to the topic. If $A$ is a commutative noetherian ring, $I$ is an ideal of $A$ and $X\in Mod(A)$, one denotes by $\Gamma_I(X)$ the set of all elements of $X$ annihilated by a power of $I$. The \textbf{local cohomology functors} are defined as the right derived functors $H^i_I(X) := \textbf{R}^i\Gamma_I(X)$. They clearly depend only on the radical $\sqrt{I}$ of $I$ rather than on $I$ itself. A key fact (see~\cite[\S5.1]{BrSh} or~\cite[\S2.1]{Hu}) is that the local cohomology of $X$ is isomorphic to the cohomology of $C_I\otimes_{A}X$, where $C_I$ is the \v{C}ech complex for any chosen finite set of generators $x_1,\dots,x_n$ of $I$:
\[ C_I = \bigotimes_{i=1}^n (A\to A_{x_i}) = \big(A \to \bigoplus_{i} A_{x_i} \to \bigoplus_{i<j} A_{x_ix_j} \to \cdots \to A_{x_1x_2\dots x_n}\big). \]
Hence, the cohomology of $C_I\otimes_{A}X$ also depends only on $\sqrt{I}$. In fact, a much finer result holds by~\cite[Proposition 6.10]{DG}. If $I,I'\subseteq A$ are ideals such that $\sqrt{I} = \sqrt{I'}$, then $C_I$ and $C_{I'}$ are isomorphic as objects of the derived category of $Mod(A)$, and the same is true for the truncated complexes $\tilde{C}_I$ and $\tilde{C}_{I'}$ obtained by erasing the leftmost term $A$.

If $A=\bigoplus_{n\in\Z}A_n$ is $\Z$-graded and $I$ is a homogeneous ideal, one defines graded local cohomology functors in the same vein, and one can again use a graded \v{C}ech complex with respect to a finite set of homogeneous generators of $I$. Both the graded \v{C}ech complex and its truncated version again only depend on $\sqrt{I}$. We refer to \cite[Chapters 12 and 13]{BrSh} for details.

For our example, we need a source of silting ring epimorphisms which are a priori not universal localisations. The idea is very easy in essence, we simply consider degree zero components of graded localisations. For the sake of completeness, we recall that, by~\cite[Theorem 1.5.5]{BrHe}, a commutative $\Z$-graded ring $A = \bigoplus_{n \in \Z} A_n$ is noetherian if and only if $A_0$ is noetherian and $A$ is a finitely generated $A_0$-algebra.

\begin{proposition}\label{degree zero flatness}
Let $A$ be a commutative noetherian $\Z$-graded ring and let $h\in A$ be a homogeneous element such that the principal ideal $Ah$ is generated in degree zero up to radical (i.e.\ $h\in\sqrt{AI}$ for $I=Ah\cap A_0$). Then the degree zero component
\[ f_0\colon A_0 \to B_0 \]
of the graded localisation $f\colon A \to B:= A_h$ is a silting ring epimorphism.
\end{proposition}

\begin{proof}
Since $I = Ah\cap A_0$ is an ideal of $A_0$, we can fix a finite collection $x_1,\dots,x_n \in A_0$ of generators of $I$. Consider the truncated graded \v{C}ech complex $\tilde{C}_{AI} = (\bigoplus_{i} A_{x_i} \to \bigoplus_{i<j} A_{x_ix_j} \to \cdots \to A_{x_1x_2\dots x_n})$ of $AI \subseteq A$. Since this complex is isomorphic in the derived category of graded $A$-modules to the truncated graded \v{C}ech complex of the principal ideal $Ah$, which is simply $\tilde{C}_{Ah} \cong A_h$, we obtain an exact sequence of graded $A$-modules
\[ 0 \to A_h \to \bigoplus_{i} A_{x_i} \to \bigoplus_{i<j} A_{x_ix_j} \to \cdots \to A_{x_1x_2\dots x_n} \to 0. \]
Restricting to degree zero, we obtain an exact sequence in $Mod(A_0)$:
\[
0 \to B_0 \to \bigoplus_{i} (A_0)_{x_i} \to \bigoplus_{i<j} (A_0)_{x_ix_j} \to \cdots \to (A_0)_{x_1x_2\dots x_n} \to 0 \eqno{(\dag)}
\] 
Since all the terms but $B_0$ are clearly flat $A_0$-modules, $B_0$ must be flat too.

The latter exact sequence tells us more, however. Since $A_0$ is noetherian and all terms but $B_0$ are countably generated, $B_0$ is countably generated as well. In particular, $B_0$ has projective dimension at most $1$ by~\cite[Lemma 1.2.8]{GT}, as it is a countable direct limit of projective modules. Moreover, \cite[Propositions 4.3 and 4.6]{DG} tell us that the canonical map of complexes $A_0 \to \tilde{C}_{I}$ induces an isomorphism $\tilde{C}_I \to \tilde{C}_I \otimes_{A_0} \tilde{C}_I$ in the derived category upon tensoring by $\tilde{C}_I$. Since $B_0 \cong \tilde{C}_I$ in the derived category by $(\dag)$, we infer that $B_0\otimes_{A_0} f_0\colon B_0 \to B_0\otimes_{A_0}B_0$ is an isomorphism, and hence that $f_0$ is a ring epimorphism. It follows from Example~\ref{silting pd1} that $f_0\colon A_0 \to B_0$ is a silting ring epimorphism.
\end{proof}

\begin{remark}
In terms of algebraic geometry, the exactness of $(\dag)$ means, by \cite[Theorem III.3.7]{Hart}, that $U = \Spec A_0 \setminus V(I)$ is an open affine subscheme of $\Spec A_0$. Then $U \cong \Spec B_0$ as schemes, since $B_0$ is the ring of global sections of $U$, and $f_0$ is simply the restriction of sections along the open immersion $\Spec B_0\to\Spec A_0$.
\end{remark}

Now we need a criterion for proving that $f_0\colon A_0 \to B_0$ is indeed not a universal localisation. This is a little more tricky and our criterion is obtained by abstraction from~\cite[Remarks 2.1]{HS}. We remind the reader of the fact that universal localisations coincide with localisations at multiplicative sets if $A_0$ happens to be local (Example~\ref{classical vs universal loc}).

\begin{proposition}\label{example not localisation abstract}
In the situation of Proposition~\ref{degree zero flatness}, suppose further that $A$ is a unique factorisation domain with all homogeneous units contained in $A_0$ and that $h$ is irreducible of non-zero degree. Then $f_0\colon A_0 \to B_0$ reflects units (i.e.\ $a\in A_0$ is a unit if $f_0(a)\in B_0$ is a unit).
In particular, $f_0$ is a localisation at a multiplicative set $\Sigma\subseteq A_0$ if and only if it is an isomorphism.
\end{proposition}

\begin{proof}
Let $a\in A_0$ be such that $f_0(a)$ is a unit in $B_0$. Then $f_0(a)$ is also a unit in $B = A_h$, but, since $A$ is a unique factorisation domain and $h$ is irreducible, all units in $B$ are of the form $u\cdot h^i$, where $i\in\Z$ and $u$ is a unit of $A$. Thus, $a=u h^i$ in $A$ for some $i\ge 0$ and, since $a$ and $h$ are homogeneous, so is $u$. It follows that $u$ is of degree zero, that $i=0$ and, in particular, that $a=u$ is a unit of $A_0$.
\end{proof}

\begin{example}[{\cite[Remarks 2.1]{HS}}]\label{example not localisation concrete}
Here is an explicit example of a (flat) tilting ring epimorphism which is not a universal localisation. 
Let $\Kbb$ be an algebraically closed field and $A = \Kbb[y_0,y_1,y_2,y_3]$ be a polynomial ring, with a grading given by
\[ \lvert y_0\rvert = \lvert y_3\rvert = 1 \qquad\textrm{and}\qquad \lvert y_1\rvert = \lvert y_2\rvert = -1. \]
We localise $A$ at the homogeneous element $h=y_0y_1^3-y_1^3y_3$ of degree $-2$ and take $f_0\colon A_0 \to (A_h)_0 = B_0$. We claim that this is a tilting ring epimorphism which is not a universal localisation. In particular, this gives an instance of Theorem~\ref{Main 1}(2) and Theorem~\ref{silting countable type} where the reduction cannot be improved to give us a set of finitely presented objects.

Indeed, notice first that $I = Ah\cap A_0$ contains the elements $hy_0^2$ and $hy_3^2$. Since $h^3 \in Ay_0^2+Ay_3^2$, we have $h^4 \in AI$. In particular, $f_0\colon A_0 \to B_0$ is a silting ring epimorphism by Proposition~\ref{degree zero flatness}. However, since $f_0$ is injective, it follows from the proof of the proposition and from Example~\ref{tilting pd1} that $f_0$ is even a tilting ring epimorphism.

Note that Proposition~\ref{example not localisation abstract}, as it stands, only tells us that $f_0$ is not a localisation at a multiplicative set. It is not difficult, however, to exclude universal localisations as well. The idea is to localise both $A$ and $B$ at a suitable multiplicative set $\Sigma\subseteq A_0$ to make $A_0$ local. Once we then show that 
\[ (f_0)_\Sigma\colon (A_0)_\Sigma \to (B_0)_\Sigma \]
is not a localisation at a multiplicative set, it will not be a universal localisation either, and so also the original map $f_0$ cannot be a universal localisation.

To explain how to choose $\Sigma$, let us briefly look at the geometry of the embedding $A_0 \to A$. As a $\Kbb$-subalgebra of $A$, $A_0$ is generated by $y_0y_1$, $y_0y_2$, $y_1y_3$ and $y_2y_3$. Moreover, the map $\Kbb[x_0,x_1,x_2,x_3]/(x_0x_3-x_1x_2) \to \Kbb[y_0,y_1,y_2,y_3]$ given by 
\[ x_0\mapsto y_0y_1,\quad x_1\mapsto y_0y_2,\quad x_2\mapsto y_1y_3,\quad x_3\mapsto y_2y_3 \]
is injective since it is the map corresponding to the surjective polynomial map of affine varieties
\begin{align*}
\mathbb{A}^4_\Kbb &\longrightarrow \{ \left(\begin{smallmatrix}x_0 & x_1 \\ x_2 & x_3 \end{smallmatrix}\right) \in M_2(\Kbb) \mid \det \left(\begin{smallmatrix}x_0 & x_1 \\ x_2 & x_3 \end{smallmatrix}\right) = 0 \}
\quad (\subseteq \mathbb{A}^4_\Kbb)
\\
(y_1,y_2,y_3,y_4) &\longmapsto \left(\begin{smallmatrix}y_0 \\ y_3\end{smallmatrix}\right) \cdot (\begin{smallmatrix}y_1 & y_2\end{smallmatrix}) = \left(\begin{smallmatrix}y_0y_1 & y_0y_2 \\ y_1y_3 & y_2y_3 \end{smallmatrix}\right).
\end{align*}
In particular, $A_0$ is the coordinate ring of the variety of $2\times 2$ singular matrices over $\Kbb$ and has a singularity at the origin. We localise at the maximal ideal of $A_0$ corresponding to the origin. That is, $\Sigma$ is the multiplicative set of all degree zero polynomials in $A=\Kbb[y_0,y_1,y_2,y_3]$ which have a non-zero absolute term.

Now $A_\Sigma$ is a unique factorisation domain since $A$ is such and $(A_\Sigma)_0=(A_0)_\Sigma$ is local by the construction. Moreover, the units of $A_\Sigma$ are of the form $\frac{u}{\sigma}$, where $\sigma\in\Sigma$ and $u$ divides an element of $\Sigma$ in $A$. However, if $u$ is homogeneous, it must be of degree zero since it has a non-zero absolute term. Thus, we can apply Proposition~\ref{example not localisation abstract} to $(f_0)_\Sigma\colon (A_0)_\Sigma \to (B_0)_\Sigma$ as described above.
\end{example}


\end{document}